\documentclass[a4paper, twoside,final]{amsart}
\usepackage[latin1]{inputenc}
\usepackage[T1]{fontenc}
\usepackage{textcomp}
\usepackage[english]{babel}
\usepackage{amsmath}
\usepackage{amsfonts}
\usepackage{amssymb}
\usepackage{amsopn}
\usepackage{amscd}
\usepackage{amsthm}
\usepackage[sc]{mathpazo}
\usepackage{mathrsfs}
\usepackage{graphicx}
\usepackage{enumerate}
\usepackage{hyperref}
\hypersetup{
  colorlinks=true,
  linkcolor=blue,
  citecolor=blue,
  linktoc=page,
  pdfauthor={Yohann de Castro}
}
\graphicspath{{Images/}}	
\renewcommand{\d}{\mathrm{d}}
\newcommand{\R}{\mathbb{R}}
\newcommand{\abs}[1]{\left\vert#1\right\vert}
\newtheoremstyle{theo}
{}
{}
{\itshape}
{\parindent}
{\bf}
{\ ---}
{.5em}
{}%
\theoremstyle{theo}
\newtheorem{theorem}{Theorem}[section]
\newtheorem{lemma}[theorem]{Lemma}
\newtheorem{proposition}[theorem]{Proposition}
\newtheorem*{cor}{Corollary}

\newtheoremstyle{def}%
{}
{}
{\itshape}
{\parindent}
{\bf}
{\ ---}
{.5em}
{}%
\theoremstyle{def}
\newtheorem{definition}{Definition}[section]

\theoremstyle{remark}
\newtheorem*{remark}{Remark}

\title{Quantitative Isoperimetric Inequalities on the Real Line}
\date{\today}
\keywords{Isoperimetric inequalities, Asymmetry, Log-concave measures, Gaussian measure.}
\author{Yohann de Castro}
\address{Institut de Math\'ematiques de Toulouse (CNRS UMR 5219). Universit\'e Paul Sabatier,
118 route de Narbonne, 31062 Toulouse, France.}
\email{yohann.decastro@math.univ-toulouse.fr}

\begin{document}
\begin{abstract}
In a recent paper {A. Cianchi}, {N. Fusco}, {F. Maggi}, and
{A. Pratelli} have shown that, in the Gauss space, a set of given measure and almost
minimal Gauss boundary measure is necessarily close to be a half-space.

Using only geometric tools, we extend their result to all symmetric log-concave measures
$\mu$ on the real line. We give sharp quantitative isoperimetric inequalities and prove that
among sets of given measure and given asymmetry (distance to half line, i.e. distance to sets
of minimal perimeter), the intervals or complements of intervals have minimal perimeter.
\end{abstract}
\maketitle
Denote $\d\gamma(t)=\exp(-t^2/2)\d t/\sqrt{2\pi}$ the standard one-dimensional Gaussian
measure. The classical Gaussian isoperimetric inequality \cite{MR0365680} states that among
sets of given measure in $(\R^n,\gamma^n)$ half spaces have minimal Gauss boundary measure.
This reads as $P_{\gamma^n}(\Omega)\geq J_\gamma\left(\gamma^n(\Omega)\right)$, where
$J_\gamma$ is optimal (and defined later on in the text). In their paper \cite{CFMP} {A. Cianchi},
{N. Fusco}, {F. Maggi}, and {A. Pratelli} have derived an improvement of the form
\[P_{\gamma^n}(\Omega)-J_\gamma(\gamma^n(\Omega))\geq
\Theta_{\gamma^n}\left(\gamma^n(\Omega),\lambda(\Omega)\right)\geq 0,\]
where $\lambda(\Omega)$ measures how far $\Omega$ is from a half-space. In their result the
dependence in $\lambda(\Omega)$ is precise, whereas the dependence in $\gamma^n(\Omega)$ is
not explicitly. In this paper we focus on the one dimensional case. Theorem 1.2 of {A. Cianchi},
{N. Fusco}, {F. Maggi}, and {A. Pratelli} gives that
\begin{equation}\label{result CFMP}
 P_\gamma(\Omega)\geq
J_\gamma(\gamma(\Omega))+\frac{\lambda(\Omega)}{C(\gamma(\Omega))}\sqrt{\log\left(1/{
\lambda(\Omega)}\right)}, 
\end{equation}
where $C(\gamma(\Omega))$ is a constant that depends only on $\gamma(\Omega)$. Our result
(given in Theorem \ref{Main}) is a sharp version of this statement, which is actually valid
for \textbf{all symmetric log-concave measures} $\mu$ on the real line. This quantitative
inequality
implies that a set of given measure and almost minimal boundary measure is necessarily "close"
to be half-line. This result stands not only for the Gaussian measure but for every measure
satisfying a natural hypothesis \ref{lower additivity} (defined later), as
proved in Theorem \ref{continuity}.

\subsection*{Organization of the paper}
The outline of the paper is as follows: the first section recalls basic properties of the
log-concave measures. The second part gives the main tool, named the
\textit{shifting lemma}, and establishes a \textit{sharp quantitative isoperimetric inequality}. In
the last section we provide (slightly weaker) estimates invoking \textit{only classical functions}.
\section{The isoperimetric inequality on the real line}
This section presents \textit{the standard isoperimetric inequality} for the log-concave measures,
and \textit{the asymmetry} which measures the gap between a given set and the sets of minimal
perimeter.
\subsection{The standard isoperimetric inequality for the log-concave measures}
Let $\mu$ be a measure with {density function $f$}. Throughout this paper, we assume
that
\begin{enumerate}[$(i)$]
    \item the function $f$ is supported and positive over some interval $(a_f,b_f)$, where
$a_f$ and $b_f$ can be infinite,
  \item the measure $\mu$ is a \textit{probability} measure,
  \item the measure $\mu$ is a \textit{log-concave measure},
    \item and the measure $\mu$ is symmetric with respect to the origin.
\end{enumerate}
Observe that the point $(iv)$ is not restrictive. As a matter of fact, the measure
$\mu(.+\alpha)$, where $\alpha\in\R$, shares the same isoperimetric properties as the measure
$\mu$. By the same token, the assumption $(ii)$ is obviously not restrictive.
\subsubsection{The $\mu$-perimeter}
Let $\Omega$ be a measurable set. Following \cite{MR0257325}, define the set $\Omega^d$ of all
points with \textit{density} exactly $d\in[0,1]$ as
\[\Omega^d=\bigg\{x\in\R,\quad\lim_{\rho\rightarrow 0}\frac{\mathcal{L}^1(\Omega\cap
B_\rho(x))}{\mathcal{L}^1(B_\rho(x))}=d\bigg\},\]
where $\mathcal L^1$ is the \textit{Lebesgue measure} over the real line and $B_\rho(x)$ the
ball with center $x$ and radius $\rho$. Define the \textbf{essential boundary}
$\partial^M \Omega$ as the set $\R\setminus\left(\Omega^0\cup\Omega^1\right)$,
consisting of points with neither empty nor full density. Define the \textbf{$\mu$-perimeter}
as
\begin{equation}\label{mu perimeter}
P_\mu(\Omega)=
\mathcal{H}_\mu^0(\partial^M\Omega)=\int_{\partial^M\Omega}f(x)\d\mathcal{H}^0(x) ,
\end{equation}
where $\mathcal{H}^0$ is the \textit{Hausdorff measure} of dimension $0$ over $\R$ and
$\mathcal{H}^0_\mu$ the measure of density $f$ with respect to $\mathcal{H}^0$. The
\textbf{isoperimetric function} $I_\mu$ of the measure $\mu$ is defined by
\begin{equation}\label{Isoperimetric function}
I_\mu(r)=\inf_{\mu(\Omega)=r}P_\mu(\Omega).
\end{equation}
In the log-concave case, we can give an explicit form to the isoperimetric function using
the {function $J_\mu$ }.
\subsubsection{The function $J_\mu$}
Denote $F$ the {cumulative distribution function} of the measure $\mu$. Since the function $f$ is
supported
and positive over some interval $(a_f,b_f)$ then the cumulative distribution function is increasing
on the interval $(a_f,b_f)$. Define
\begin{equation}\label{jmu function}
 J_\mu(r)=f\big(F^{-1}(r)\big),
\end{equation}
where the quantile $r$ ranges strictly from $0$ to $1$, $J_\mu(0)=J_\mu(1)=0$, and $F^{-1}$ denotes
the inverse function of $F$.
\subsubsection{The standard isoperimetric inequality}
Following the article \cite{MR1327260} of S. G. Bobkov, since the measure $\mu$ is symmetric
with respect to the origin, then the inverse function of $F$ satisfies,
\begin{equation}\label{inverse function}
 F^{-1}(r)=\int_{1/2}^r\frac{\d t}{J_\mu(t)},\quad\forall r\in(0,1).
\end{equation}
Using (\ref{inverse function}), one can check \cite{MR1327260} that the measure
$\mu$ is log-concave \textbf{if and only if} $J_\mu$ is concave on $(0,1)$. Furthermore it is known
\cite{MR0399402} that the infima of (\ref{Isoperimetric function}) are exactly the intervals
$(-\infty,\sigma_-)$ and $(\sigma_+,+\infty)$, where $\sigma_-=F^{-1}(r)$ and
$\sigma_+=F^{-1}(1-r)$. The \textbf{isoperimetric inequality} states
\begin{equation}\label{Isoperimetric inequality}
P_\mu(\Omega)\geq J_\mu(\mu(\Omega)),                                                      
\end{equation}
where $\Omega$ is a Lebesgue measurable set. This shows that, in the log-concave case, the
isoperimetric function coincides with the function $J_\mu$.

\subsection{The asymmetry}
We concern with quantifying the difference between any measurable set $\Omega$ and an
isoperimetric infimum (i.e. any measurable set such that the isoperimetric inequality
\eqref{Isoperimetric inequality} is an equality) with the same $\mu$-measure. Following
\cite{CFMP}, define
the \textbf{asymmetry} $\lambda_\mu(\Omega)$ of a set $\Omega$ as
\begin{equation}\label{asymmetry}
\lambda_\mu(\Omega)=\min\left\{\mu\!\left(\Omega\Delta(-\infty,\sigma_-)\right),\
\mu\!\left(\Omega\Delta (\sigma_+,+\infty)\right)\right\},
\end{equation}
where $\sigma_-=F^{-1}(\mu(\Omega))$ and $\sigma_+=F^{-1}(1-\mu(\Omega))$, and $\Delta$ is
\textit{the symmetric difference operator}. 
\begin{remark}
The name asymmetry \cite{MR2456887} is inherited from the case of the
Lebesgue measure on $\R^n$. In this case, the sets with minimal perimeter are balls, hence very
symmetric. 
\end{remark}
\noindent Define the \textbf{isoperimetric projection} of a set $\Omega$ as the open
half-line achieving the minimum in (\ref{asymmetry}). In the case where this minimum is not unique
we can chose whatever infima as an isoperimetric projection.
 \section{Sharp quantitative isoperimetric inequalities}
This section gives a sharp improvement of (\ref{Isoperimetric inequality})
involving the asymmetry $\lambda(\Omega)$. In their paper \cite{CFMP} {A. Cianchi}, {N. Fusco},
{F. Maggi}, and {A. Pratelli} use a technical lemma (Lemma 4.7, Continuity Lemma) to complete their
proof. Their lemma applies in the $n$-dimensional case and is based on a compactness argument
derived from powerful results in geometric measure theory. In the one-dimensional case, our
approach is \textbf{purely geometric} and does not involve the continuity lemma.

\subsection{The shifting lemma}
The shifting lemma plays a key role in our proof. This lemma was introduced in \cite{CFMP} for the
Gaussian measure. It naturally extends to even log-concave probability measures. For sake of
readability, we begin with the shifting property.
\begin{definition}[The shifting property]
We say that a measure $\nu$ satisfies the
\textbf{shifting property} when for every open interval $(a,b)$, the following is true:
\begin{enumerate}
 \item[-]If $a+b\geq0$ then for every $(a',b')$ such that $a\leq a'<b'\leq +\infty$ and
$\nu((a,b))=\nu((a',b'))$, it holds $P_\nu((a,b))\geq P_\nu((a',b'))$. In other words, if an
interval is more to the right of $0$, shifting it to the right with fixed measure, does not
increase the perimeter.
\item[-]If $a+b\leq0$ then for every $(a',b')$ such that $-\infty\leq
a'<b'\leq b$ and
$\nu((a,b))=\nu((a',b'))$, it holds $P_\nu((a,b))\geq P_\nu((a',b'))$. In other words, if an
interval is more to the left of $0$, shifting it to the left with fixed measure, does not
increase the perimeter. 
\end{enumerate}
\end{definition}
\noindent The following remark states that the shifting property can be equivalently formulated
with the complement sets.
\begin{remark}
As the perimeter is complement-invariant, we may also shift "holes". The shifting
property is equivalent to the following property.
\begin{enumerate}
 \item[-] If $a+b\geq0$ then for every $(a',b')$ such that $a\leq a'<b'\leq +\infty$ and
$\nu((a,b))=\nu((a',b'))$, it holds $P_\nu((-\infty,a)\cup(b,+\infty))\geq
P_\nu((-\infty,a')\cup(b',+\infty))$.
\item[-] If $a+b\leq0$ then for every $(a',b')$ such that $-\infty\leq
a'<b'\leq b$ and
$\nu((a,b))=\nu((a',b'))$, it holds $P_\nu((-\infty,a)\cup(b,+\infty))\geq
P_\nu((-\infty,a')\cup(b',+\infty))$.
\end{enumerate}
\end{remark}
\noindent Roughly, the next lemma shows that, for all measures such that the assumptions $(i)$,
$(ii)$, and $(iv)$ hold, the assumption $(iii)$ is equivalent to the shifting property.
\begin{lemma}[The shifting lemma]\label{shiftinglemma}
Every log-concave probability measure symmetric with respect to the origin has the shifting
property.

\noindent \textbf{---} Conversely, let $f$ be a continuous function, positive on an open interval
and null outside. If the probability measure with density function $f$ is symmetric with
respect to the origin and enjoys the shifting property then it is log-concave.
\end{lemma}
\begin{proof}
Let $x$, $r$ be in $(0,1)$ and $t$ be in $(r/2,\ 1-r/2)$. Define
$\varphi(t)=J_\mu(t-r/2)+J_\mu(t+r/2)$. It represents the $\mu$-perimeter
of $(F^{-1}(t-r/2),F^{-1}(t+r/2))$ with measure equal to $r$. The function
$J_\mu$ is symmetric with respect to $1/2$ since the density function $f$ is supposed to be
symmetric. As the function $J_\mu$ is concave and symmetric with respect to $1/2$, so is the
function $\varphi$. In particular $\varphi$ is
non-decreasing on $(r/2,1/2]$ and non-increasing on $[1/2,1-r/2)$. This gives the shifting
property.

\noindent Conversely, let $f$ be a continuous function, positive on an open interval
and null outside. Define the isoperimetric function $J_\mu$ as in (\ref{jmu
function}). We recall that $\mu$
is log-concave if and only if $J_\mu$ is concave on $(0,1)$. Since the function
$J_\mu$ is continuous, it is sufficient to have $J_\mu(x)\geq(1/2)
\left(J_\mu(x-d)+J_\mu(x+d)\right)$, for all $x\in(0,1)$, where $d$ is small enough to get
$x-d\in(0,1)$ and $x+d\in(0,1)$. Let $x$ and
$d$ be as in the previous equality. Since $\mu$ is symmetric, assume that $x\leq1/2$.
Put $a=F^{-1}(x)$, $b=F^{-1}(1-x)$, ${a'}=F^{-1}(x+d)$, ${b'}=F^{-1}(1-x+d)$, 
then $(a',b')$ is a shift to the right of $(a,b)$. By the shifting property,
we get $P_\mu((a,b))\geq P_\mu((a',b'))$. The function $J_\mu$ is symmetric with respect to
$1/2$, it yields (see Figure \ref{reciproque}),
\begin{equation*}
 \begin{array}{lclclcl}
  P_\mu((a,b)) & = & J_\mu(x)+J_\mu(1-x) & = & 2J_\mu(x) ,\\
  P_\mu((a',b'))& = & J_\mu(x+d)+J_\mu(1-x+d) & = &
J_\mu(x+d)+J_\mu(x-d).
 \end{array}
\end{equation*}
This ends the proof.
\end{proof}

\begin{figure}[!ht]
\centering
 \includegraphics[width=7cm]{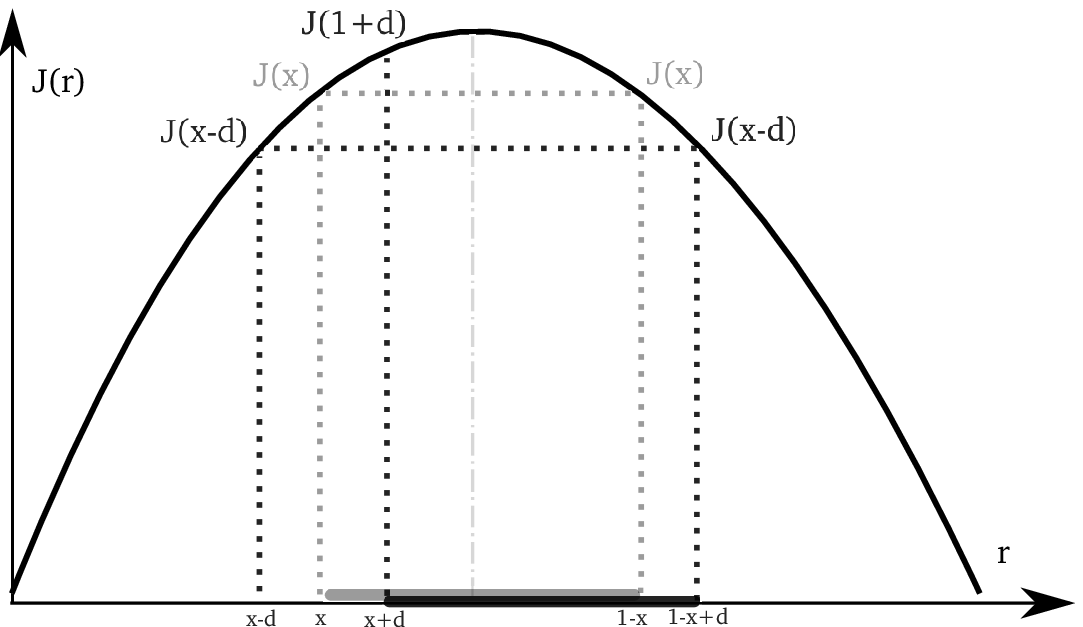}
\caption{The log-concavity is equivalent to the shifting property}\label{reciproque}
\end{figure}


\subsection{Lower bounds on the perimeter}
In the following, we show that among sets of given measure and given asymmetry, the
intervals or complements of intervals have minimal perimeter.
\subsubsection{Structure of the sets with finite perimeter}
Let $\Omega$ be a set of finite $\mu$-perimeter. Consider $(K_k)_{k\in\mathbb N}$ a
sequence of compact sets such that, for all $k\geq0$,  $K_0\subset\ldots\subset
K_k\subset(-a_f,a_f) $ and $\cup_{k\in\mathbb N} K_k=(-a_f,a_f)$. Then, it
yields
\begin{equation}\label{tight}
 \Omega=\Big(\bigcup_{k\in\mathbb N}(\Omega\cap K_k)\Big)\bigcup E,
\end{equation}
where $E$ is such that $\mu(E)=0$. Let $k$ be an integer. On the compact $K_k$ the
function $f$ is bounded from below by a positive real. Thus if $\Omega\cap K_k$ has finite
$\mu$-perimeter, so it has finite Lebesgue perimeter. As mentioned in \cite{MR1857292,
MR0257325},
one knows that every set with finite Lebesgue perimeter can be written as at most countable union
of open intervals and a set of measure equal to zero. It holds
\begin{equation*}
 \Omega\cap K_k = \Big(\bigcup_{n\in I_k}
(a_n,b_n)\Big)\bigcup\mathcal E_k,
\end{equation*}
where $I_k$ is at most countable, $\mathcal E_k$ is such that $\mu(\mathcal E_k)=0$, and
$(a_n,b_n)$ is such that
\begin{equation}\label{disjoint compact}
 d\Big((a_n,b_n),\bigcup_{l\in I_k\setminus\{n\}} (a_l,b_l)\Big)>0,
\end{equation}
for all $n$ in $I_k$ and $d$ the euclidean distance over the real line. Denote
$1\!\!1_\Omega$ the indicator function of $\Omega$ and $1\!\!1_\Omega'$ its
distribution derivative. The property
(\ref{disjoint compact}) is a consequence of the fact that $1\!\!1_\Omega'$ is locally
finite (see \cite{MR0257325} for instance). Since $K_k$ is compact, the set $I_k$ is finite. One
can check that the decomposition
(\ref{tight}) becomes
\begin{equation*}
 \Omega = \Big(\bigcup_{n\in I} (a_n,b_n)\Big)\bigcup\mathcal E,
\end{equation*}
where $I$ is at most countable, $\mathcal{E}$ such that $\mu(\mathcal E)=0$, and $(a_n,b_n)$ such
that
\begin{equation}\label{disjoint}
 d\Big((a_n,b_n),\bigcup_{k\in I\setminus\{n\}} (a_k,b_k)\Big)>0,
\end{equation}
for all $n$ in $I$. Notice that $\mu(\mathcal E)=0$. Without loss of
generality, assume that $\Omega$ is an \textit{at most countable union of open intervals such that
$1\!\!1_\Omega'$ is locally finite}.

\subsubsection{Preliminaries}
Let $\Omega$ be a set of finite $\mu$-perimeter. As mentioned previously, assume that
\begin{equation*}
 \Omega = \bigcup_{n\in I} (a_n,b_n)
\end{equation*}
where $I$ is an at most countable set and (\ref{disjoint}) holds. Suppose that
\begin{itemize}
 \item an isoperimetric
projection of $\Omega$ is $(-\infty,\sigma_-)$ (using a symmetry with respect to the origin if
necessary),
  \item and that the measure of $\Omega$ is at most $1/2$ (and we will see at the end of this
section how to extend our result to larger measures).
\end{itemize}
Then the real number $\sigma_-=F^{-1}(\mu(\Omega))$ is non-positive. Denote
$\sigma=-\sigma_-$. Since $1\!\!1_\Omega'$ is locally finite, there exists a finite
number of sets $(a_n,b_n)$ included in $(-\sigma,\sigma)$, it follows that
\begin{equation*}
 \Omega=\Big(\bigcup_{h\in\Lambda_-} A_h\Big)\cup
I\cup
\Big(\bigcup_{h=1}^{N_-} A'_h\Big)\cup\Big(\bigcup_{h=1}^{N_+} B'_h\Big) \cup
J\cup
\Big(\bigcup_{h\in\Lambda_+} B_h\Big),
\end{equation*}
\noindent where 
\begin{enumerate}
  \item[$\bullet$]
$\Lambda_-$ and $\Lambda_+$ are at most countable sets;
  \item[$\bullet$]
$A_h=(\alpha_{A_h},\beta_{A_h})$  with $\beta_{A_h}\leq-\sigma$ ($\alpha_{A_h}$ can be
infinite); 
  \item[$\bullet$] 
$I$ is either empty or of the form $I=(\alpha_I,\beta_I)$
with $\alpha_{I}\leq-\sigma<\beta_{I}$;
  \item[$\bullet$] 
$A'_h$ is either empty or of the form $A'_h=(\alpha_{A'_h},\beta_{A'_h})$ with
$-\sigma<\alpha_{A'_h}$ and $\alpha_{A'_h}+\beta_{A'_h}< 0$;
  \item[$\bullet$]
$B'_h$ is either empty or of the form $B'_h=(\alpha_{B'_h},\beta_{B'_h})$ with
$\beta_{B'_h}<\sigma$ and $\alpha_{B'_h}+\beta_{B'_h}\geq 0$;
  \item[$\bullet$] 
$J$ is either empty or of the form $J=(\alpha_J,\beta_J)$
with $\alpha_{J}<\sigma\leq\beta_{J}$;
  \item[$\bullet$] 
and $B_h$ is either empty or of the form $B_h=(\alpha_{B_h},\beta_{B_h})$ with
$\alpha_{B_h}\geq\sigma$ ($\beta_{B_h}$ can be infinite).
\end{enumerate}
\noindent From $\Omega$ we build $\Omega_0$ with same measure, same asymmetry, same
isoperimetric projection, and lower or equal perimeter. Denote $L=\bigcup_{h\in\Lambda_-}
A_h$ and $A_0=(-\infty,\beta_{A_0})$ where $\beta_{A_0}=F^{-1}(\mu(L))$. Since
$\mu(L)\leq \mu(\Omega)$, then $\beta_{A_0}\leq-\sigma$. Using the
isoperimetric inequality (\ref{Isoperimetric inequality}) with $L$, it follows that 
$P_\mu(A_0)\leq P_\mu(L)$. The same reason gives
that there exist a real number $\alpha_{B_0}\geq\sigma$ and a set
$B_0=(\alpha_{B_0},+\infty)$ with lower or equal perimeter than $\cup_{h\in\Lambda_+}
B_h$ (if non-empty). Shift to the left the intervals $A'_h$ until they reach $I$ or
$-\sigma$. Shift to the right the intervals $B'_h$
until they reach $J$ or $\sigma$. The above operation did not change the amount of mass on
left of $-\sigma$ and on the right of $\sigma$. We build a set $\Omega_0$ with same asymmetry
and same isoperimetric projection as $\Omega$ and lower or equal
perimeter,
\begin{equation*}
 \Omega_0=A_0\cup I_0\cup J_0\cup B_0,
\end{equation*}
where 
\begin{enumerate}
  \item[$\bullet$] $A_0=(-\infty,\beta_0)$  with $\beta_{A_0}\leq-\sigma$; 
  \item[$\bullet$] $I_0$ is either empty or of the form $I_0=(\alpha_{I_0},\beta_{I_0})$
with $\alpha_{I_0}\leq-\sigma<\beta_{I_0}$;
  \item[$\bullet$] $J_0$ is either empty or of the form $J_0=(\alpha_{J_0},\beta_{J_0})$
with $\alpha_{J_0}<\sigma\leq\beta_{J_0}$;
  \item[$\bullet$] and $B_0$ is either empty or of the form $B_0=(\alpha_{B_0},+\infty)$ with
$\alpha_{B_0}>\sigma$.
\end{enumerate}

\begin{figure}[!ht]
\centering
 \includegraphics[width=10cm]{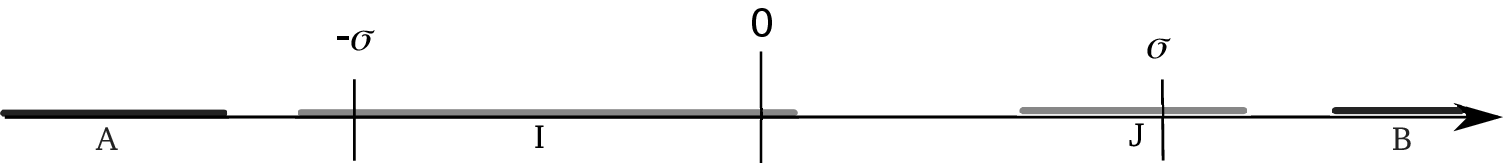}
\caption{The set $\Omega_0$}
\label{Omega0}
\end{figure}
 \subsubsection{Lowering the perimeter}
A case analysis on the non-emptiness of sets $I_0$ and $J_0$ is required to obtain the claimed
result. Every step described below lowers the perimeter (thanks to the shifting lemma, Lemma
\ref{shiftinglemma}) and preserves the asymmetry. Before exposing this, we recall that the set
$\Omega_0$ is supposed to have $(-\infty,-\sigma)$ as an isoperimetric projection. Thus we pay
attention to the fact that it is totally equivalent to ask either the asymmetry to be preserved or
the quantity $\lambda(\Omega_0)/2=\mu(\Omega_0\cap(-\infty,-\sigma))$ to be
preserved trough all steps described below.
\begin{description}
\item[\textbf{If $I_0$ and $J_0$ are both nonempty}] 
Applying a symmetry with respect to the origin if necessary, assume that the center of mass of the
hole
between $I_0$ and $J_0$ is not less than $0$. We can shift this hole to the right until it touches
$\sigma$. Using the isoperimetric inequality (\ref{Isoperimetric inequality}), assume
that there exist only one interval of the form $(\alpha_{B_0}',+\infty)$ on the right of
$\sigma$. We get the case where $I_0$ is nonempty and $J_0$ is empty.
\item[\textbf{If $I_0$ is nonempty and $J_0$ is empty}]
Shift the hole between $A_0$ and $I_0$ to the left until $-\infty$ (there exists one and
only one hole between $A_0$ and $I_0$ since $\Omega_0$ is not a full measure set of
$(-\infty,-\sigma)$). We shift the hole between $I_0$ and $B_0$ to the right until $+\infty$ (one
readily checks that its center of mass is greater than 0). We get the only interval with same
asymmetry and same isoperimetric projection as the set $\Omega_0$. This interval is of the form
(the letter $c$ stands for connected),
\begin{equation}\label{OmegaC}
 \Omega_c:=\big(F^{-1}\left({\lambda(\Omega_0)}/2\right)\,,\
F^{-1}\left(\mu(\Omega_0)+{\lambda(\Omega_0)}/2\right)\big).
\end{equation}
\item[\textbf{If $J_0$ is nonempty and $I_0$ is empty}]
Shift to the right the hole between $J_0$ and $B_0$ to $+\infty$ (there exists one
hole between $J_0$ and $B_0$ since $\Omega_0$ is not a full measure set of
$(\sigma,+\infty)$). We obtain a set $A_0\cup J'$ where $J'$ is a neighborhood of $\sigma$.
\begin{enumerate}
 \item [$\bullet$]If $\mu(J')>\mu(A_0)$, then
shift $J'$ to the right
(which has center of mass greater than 0) till
$J'\cap(\sigma,+\infty)$ has weight equal to
$\mu(A_0)$ (in order to preserve asymmetry). Using a reflection in respect to
the origin, we find ourselves in the
case where $I_0$ is nonempty and $J_0$ is empty.
\item [$\bullet$]If $\mu(J')\leq\mu(A_0)$, then shift $J'$ (which has
center of mass greater than 0) to the right until $+\infty$ and get the case where $I_0$ and $J_0$
are both empty.
\end{enumerate}
\item[\textbf{If $I_0$ and $J_0$ are both empty}]
Then the set $\Omega_0$ is of the form ($d$ stands for
disconnected),
\begin{equation}\label{OmegaD}
 \Omega_d=\big(-\infty\,,\
F^{-1}\left(\mu(\Omega_0)-{\lambda(\Omega_0)}/2\right)\big) \cup
\big(F^{-1}\left(1-{\lambda(\Omega_0)}/2\right)\,,\ +\infty\big).
\end{equation}
\end{description}
\noindent Thus we proved the following lemma.
\begin{lemma}\label{resultcaseanalysis}
 Let $\Omega$ be a measurable set with $\mu$-measure at most $1/2$ and $\lambda(\Omega)$
be the asymmetry of $\Omega$. Then, it holds
\begin{equation*}
 P_\mu(\Omega)\geq P_\mu\left(\Omega_c\right)\ \mathrm{or}\ P_\mu(\Omega)\geq
P_\mu\left(\Omega_d\right),
\end{equation*}
where 
\begin{enumerate}
 \item[$\bullet$] $\Omega_c=\big(F^{-1}\big(\frac{\lambda(\Omega)}2\big)\,,\
F^{-1}\big(\mu(\Omega)+\frac{\lambda(\Omega)}2\big)\big),$
  \item[$\bullet$] $\Omega_d=\big(-\infty\,,\
F^{-1}\big(\mu(\Omega)-\frac{\lambda(\Omega)}2\big)\big) \bigcup
\big(F^{-1}\big(1-\frac{\lambda(\Omega)}2\big)\,,\ +\infty\big),$
\end{enumerate}
are sets with same measure and same asymmetry as $\Omega$.
\end{lemma}
\noindent Let us emphasize that the sets $\Omega_c$ and $\Omega_d$ have fixed isoperimetric
projection (i.e. $(-\infty,-\sigma)$), asymmetry, and measure. Observe that these properties are
satisfied only
for particular values of $\mu(\Omega)$ and $\lambda(\Omega)$.

\subsubsection{Domains of sets with minimal perimeter given measure and given asymmetry}
We are concerned here with the domain $D=\left\{(\mu(\Pi),\lambda(\Pi)),\ \Pi\
\mathrm{measurable\ set}\right\}$. The next lemma shows that
asymmetry and perimeter are complement invariant. 
\begin{lemma}\label{complement}
For every sets $A$ and $B$ with finite $\mu$-perimeter the following is true:
\begin{enumerate}
  \item[$\bullet$] the symmetric difference is complement-invariant: $A\Delta B=A^c\Delta B^c$,
  \item[$\bullet$] the asymmetry is complement-invariant: $\lambda(A)=\lambda(A^c)$,
  \item[$\bullet$] the perimeter is complement-invariant: $P_\mu(A)=P_\mu(A^c)$,
 \item[$\bullet$] and it holds $m(A)=m(A^c)$ where $m(A)=\min\left\{\mu(A),\ 1-\mu(A)\right\}$.
\end{enumerate}
\end{lemma}
\begin{proof} One can check that the symmetric difference is complement-invariant
(remark that $1\!\!1_{A\Delta B}=\abs{1\!\!1_{A}-1\!\!1_{B}}$). The
essential boundary is also complement-invariant, thus Definition \ref{mu perimeter}
shows that the $\mu$-perimeter is complement-invariant. Considering the symmetry of the
isoperimetric function $J_\mu$, we pretend that the isoperimetric projections are complements
of the isoperimetric projections of the complement. This latter property and the fact that
the symmetric difference is complement-invariant give that the asymmetry is
complement-invariant. The last equality is easy to check since $\mu$ is a probability
measure.
\end{proof}
\noindent Since asymmetry is complement-invariant, the domain $D$ is symmetric with respect to the
axis $x=1/2$. Furthermore, we have the next lemma.
\begin{lemma}\label{Domain}
 For every measurable set $\Pi$, $0\leq\lambda(\Pi)\leq \min\left(2\,m(\Pi)\,,\,1-m(\Pi)\right)$,
where $m(\Pi)=\min\left\{\mu(\Pi)\,,\ 1-\mu(\Pi)\right\}$.
\end{lemma}
\begin{proof}
Let $\Pi$ be a measurable set. As asymmetry $\lambda(\Pi)$ and $m(\Pi)$ are
complement-invariant (see Lemma \ref{complement}), suppose that $\mu(\Pi)\leq1/2$ thus
$m(\Pi)=\mu(\Pi)$. Using symmetry with respect to the origin, suppose that
$(-\infty,-\sigma)$ is an isoperimetric projection of $\Pi$ (where
$\sigma=-F^{-1}(\mu(\Pi)$).

\noindent We begin with the inequality $\lambda(\Pi)\leq 1-\mu(\Pi)$. Since $(-\infty,-\sigma)$ is
an
isoperimetric projection of $\Pi$, it holds
\[\mu(\Pi\cap(\sigma,+\infty))\leq\mu(\Pi\cap(-\infty,-\sigma))=\mu(\Pi)-{\lambda(\Pi)}/2.\]
Remark that
$\mu((-\sigma,\sigma))=1-2\,\mu(\Pi)$. Hence,
$\lambda(\Pi)/2=\mu(\Pi\cap(-\sigma,+\infty))\leq1-2\,\mu(\Pi)+\mu(\Pi)-\lambda(\Pi)/2$,
which gives the expected result.

\noindent The inequality $\lambda(\Pi)\leq 2\,\mu(\Pi)$ can be deduced from
\[{\lambda(\Pi)}/2=\mu((-\infty,-\sigma)\setminus\Pi)\ \mathrm{and}\
\mu((-\infty,-\sigma)\setminus\Pi)\leq\mu((-\infty,-\sigma))=\mu(\Pi).\] It is
clear that $\lambda(\Pi)\geq0$, this ends the proof.
\end{proof}
\noindent By construction, the sets $\Omega_c$ and $\Omega_d$ verify three properties:
\begin{enumerate}
 \item their measure is $\mu(\Omega)$,
  \item their asymmetry is $\lambda(\Omega)$,
  \item their isoperimetric projection is $(-\infty,-\sigma)$.
\end{enumerate}
We recall that $\mu(\Omega)\leq1/2$. Using Lemma \ref{Domain}, it is easy to check
that $\Omega_c$ satisfies these properties if and only if
\begin{equation}\label{constraints omega c}
  0\leq\lambda(\Omega)\leq\min(2\,\mu(\Omega),\,1-\mu(\Omega)).
\end{equation}
Using the definition of the isoperimetric projection, one can check that $\Omega_d$ satisfies
these properties \textbf{if and only if} 
\begin{equation}\label{constraints omega d} 
  0\leq\lambda(\Omega)\leq \mu(\Omega).
\end{equation}
Notice that on domain $0\leq\lambda(\Omega)\leq \mu(\Omega)$ both sets exist. On this domain,
\begin{equation*}
 P_\mu(\Omega_d) - P_\mu(\Omega_c) =  J_\mu\big(\mu(\Omega)-{\lambda(\Omega)}/2\big) -
J_\mu\big(\mu(\Omega)+{\lambda(\Omega)}/2\big).
\end{equation*}
Since $\mu(\Omega)-{\lambda(\Omega)}/2\leq \mu(\Omega)+{\lambda(\Omega)}/2\leq
1-\mu(\Omega)+{\lambda(\Omega)}/2$, we deduce from the concavity and the symmetry of
the isoperimetric function that $P_\mu(\Omega_d)\leq P_\mu(\Omega_c)$ with equality \textbf{if and
only if}$\mu(\Omega)=1/2$. Using (\ref{constraints omega c}) and (\ref{constraints omega d}), we
have the following result.
\begin{lemma}\label{lemmeexistenceomegacd}
 Let $\Omega$ be a measurable set with $\mu$-measure at most $1/2$ and $\lambda(\Omega)$
be the asymmetry of $\Omega$. Then
\begin{enumerate}
 \item[$\bullet$] the connected set of the form
$\Omega_c=\big(F^{-1}\big({\lambda(\Omega)}/2\big)\,,\
F^{-1}\big(\mu(\Omega)+{\lambda(\Omega)}/2\big)\big)$ has same measure and same asymmetry
as $\Omega$ when $0<\lambda(\Omega)\leq1-\mu(\Omega)$,
  \item[$\bullet$] and the disconnected set of the form \[\Omega_d=\big(-\infty\,,\
F^{-1}\big(\mu(\Omega)-{\lambda(\Omega)}/2\big)\big) \cup
\big(F^{-1}\big(1-{\lambda(\Omega)}/2\big)\,,\ +\infty\big)\] has same asymmetry
and same measure as $\Omega$ when $0<\lambda(\Omega)\leq \mu(\Omega)$.
\end{enumerate}
Besides, on the domain $0<\lambda(\Omega)\leq \mu(\Omega)$, $P_\mu(\Omega_d)\leq P_\mu(\Omega_c)$
with equality \textbf{if and only if} $\mu(\Omega)=1/2$.
\end{lemma}
\subsubsection{Isoperimetric deficit}
We end our case analysis with the following important result.
\begin{theorem}\label{Isoperimetric Deficit}\label{Main}
Let $\Omega$ be a measurable set and $\lambda(\Omega)$ be the asymmetry of $\Omega$.
Set $m(\Omega)=\min\left\{\mu(\Omega)\,,\ 1-\mu(\Omega)\right\}$.
\begin{enumerate}
  \item[$\bullet$] If $0<\lambda(\Omega)\leq m(\Omega)$ then
\begin{equation}\label{minorationOmegaD}
 P_\mu(\Omega)\geq J_\mu\big(m(\Omega) -{\lambda(\Omega)}/2\big)+J_\mu\big({\lambda(\Omega)
}/2\big),
\end{equation}
  \item[$\bullet$] If $ m(\Omega)<\lambda(\Omega)\leq\min(2\,m(\Omega),\,1-m(\Omega))$
then
\begin{equation}\label{minorationOmegaC}
 P_\mu(\Omega)\geq J_\mu\big(m(\Omega)+{\lambda(\Omega)}/2\big)+J_\mu\big({\lambda(\Omega)}/
2\big),
\end{equation}
and these inequalities are sharp.
\end{enumerate}
\end{theorem}
\begin{proof}
Let $\Omega$ be a measurable set. If $\Omega$ has infinite $\mu$-perimeter the result is
true, hence
assume that $\Omega$ has finite $\mu$-perimeter. We distinguish four cases as illustrated in
Figure \ref{Domains}.

\begin{figure}[!ht]
\centering
 \includegraphics[height=4.5cm]{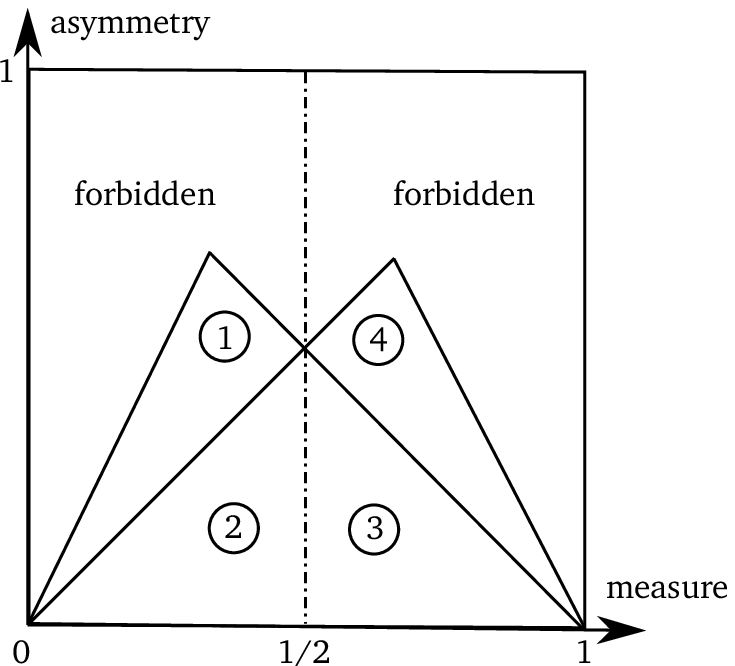}
\caption{Domains of the sets with minimal perimeter given measure
and asymmetry}\label{Domains}\label{Domaindef}
\end{figure}

\noindent \textbf{If 
$\Omega$ has measure at most $1/2$}, then $m(\Omega)=\mu(\Omega)$. Consider sets $\Omega_c$
defined in \eqref{OmegaC} and $\Omega_d$ defined in \eqref{OmegaD}, compute
\begin{equation}\label{optimaux}
\begin{array}{c}
 P_\mu\big(\Omega_d\big)=J_\mu\big(\mu(\Omega)
-{\lambda(\Omega)}/2\big)+J_\mu\big({ \lambda(\Omega) }/2\big),\\
P_\mu(\Omega_c)=J_\mu\big(\mu(\Omega)+{\lambda(\Omega)}/2\big)+J_\mu\big({\lambda(\Omega)}/
2\big).
\end{array}
\end{equation}
Lemma \ref{resultcaseanalysis} says that $\Omega$ has greater or equal $\mu$-perimeter than
$\Omega_c$ or $\Omega_d$.
\begin{description}
\item[\textbf{Domain 1}] If $\mu(\Omega)<\lambda(\Omega)\leq 1-\mu(\Omega)$ (and thus
$m(\Omega)<\lambda(\Omega)\leq 1-m(\Omega)$) then from Lemma
\ref{lemmeexistenceomegacd} we know that $\Omega_d$ does not exist for such range of
asymmetry. Necessary, it follows that $P_\mu(\Omega)\geq P_\mu(\Omega_c)$. Using
(\ref{optimaux}), we complete (\ref{minorationOmegaC}).
\item[\textbf{Domain 2}] If $0<\lambda(\Omega)\leq \mu(\Omega)$ (and thus
$0<\lambda(\Omega)\leq m(\Omega)$) then from Lemma
\ref{lemmeexistenceomegacd} we know that $P_\mu(\Omega_d)\leq P_\mu(\Omega_c)$. Thus
$P_\mu(\Omega)\geq P_\mu(\Omega_d)$. Using (\ref{optimaux}), we get
(\ref{minorationOmegaD}).
\end{description}
\noindent \textbf{If 
$\Omega$ has measure greater than $1/2$}, then $1-\mu(\Omega)=m(\Omega)$. The Lemma
\ref{complement}
shows how to deal with
sets of large
measure and allows us to consider either $\Omega$ or its complement.
\begin{description}
\item[\textbf{Domain 3}] If $0<\lambda(\Omega)\leq 1-\mu(\Omega)$ (and thus
$0<\lambda(\Omega)\leq m(\Omega)$), the complement of
$\Omega$ satisfies $0<\lambda(\Omega^c)\leq \mu(\Omega^c)$ (Domain 2). Thus we know that
$P_\mu(\Omega_d)\leq P_\mu(\Omega_c)$ (see the previous case on
Domain 2). Finally, $P_\mu(\Omega)\geq P_\mu\left(\Omega_d^c\right)$ where
$\Omega_d$ has same asymmetry and measure equal to $m(\Omega)$. Using (\ref{optimaux}), we
complete (\ref{minorationOmegaD}).
\item[\textbf{Domain 4}] If $1-\mu(\Omega)<\lambda(\Omega)\leq \mu(\Omega)$ (and thus
$m(\Omega)<\lambda(\Omega)\leq 1-m(\Omega)$), the complement of
$\Omega$ satisfies $\mu(\Omega^c)<\lambda(\Omega^c)\leq 1-\mu(\Omega^c)$ (Domain 1). From
the case on Domain 1, we know that $P_\mu\left(\Omega^c\right)\geq P_\mu(\Omega_c)$. Thus,
$P_\mu(\Omega)\geq P_\mu\left(\Omega_c^c\right)$ where
$\Omega_c$ has same asymmetry and measure equal to $m(\Omega)$. Using (\ref{optimaux}), we
get (\ref{minorationOmegaC}).
\end{description}
This case analysis ends the proof.
\end{proof}
\noindent The equalities $(\ref{optimaux})$ and the case analysis of the proof of Theorem
\ref{Isoperimetric Deficit} give the explicit lower bounds on $\mu$-perimeter. 
\begin{proposition}\label{Explicit Sets}
The sets $($see Figure \ref{Domains}$)$
\begin{enumerate}
  \item[$\bullet$] $\Omega_c=\big(F^{-1}\big(\frac{\lambda}2\big)\,,\
F^{-1}\big(\mu+\frac{\lambda}2\big)\big)$, with
$0<\mu<\lambda\leq 1-\mu$ and $\mu\leq1/2$ $($Domain 1$)$,
  \item[$\bullet$] $\Omega_d=\big(-\infty\,,\
F^{-1}\big(\mu-\frac{\lambda}2\big)\big) \cup
\big(F^{-1}\big(1-\frac{\lambda}2\big)\,,\ +\infty\big)$, with
$0<\lambda\leq \mu$ and $\mu\leq1/2$ $($Domain 2$)$,
  \item[$\bullet$]
$\Omega_d^c=\big(F^{-1}\big(1-\mu-\frac{\lambda}2\big)\,,\
F^{-1}\big(1-\frac{\lambda}2\big)\big)$, with $0<\lambda\leq
1-\mu$ and $1/2\leq\mu<1$ $($Domain 3$)$,
  \item[$\bullet$] $\Omega_c^c=\big(-\infty\,,\
F^{-1}\big(\frac{\lambda}2\big)\big) \bigcup
\big(F^{-1}\big(1-\mu+\frac{\lambda}2\big)\,,\ +\infty\big)$, with
$1-\mu<\lambda\leq \mu$ and $1/2\leq\mu<1$ $($Domain 4$)$,
\end{enumerate}have the
lowest perimeter given measure $\mu$ and given asymmetry $\lambda$.
\end{proposition}
\begin{remark}
The proof of Proposition \ref{Explicit Sets} shows that, up to a negligible set,
$\Omega_c$, $\Omega_d$, $\Omega_d^c$ and $\Omega_c^c$ are optimal given measure and given
asymmetry. Moreover, it shows that the bounds in Theorem
\ref{Isoperimetric Deficit} are sharp.
\end{remark}
 \section{Sharp estimate on the asymmetry}
In this section we use Theorem \ref{Main} to get a sharp estimate on the asymmetry. As a matter of
fact, we show that a set of given measure and almost minimal boundary measure is
necessarily close to be half-line.

\subsection{The isoperimetric deficit function}
We concern with an upper bound on the asymmetry of sets of given measure and given
perimeter. Let $\Omega$ be a set with finite perimeter. Define the \textbf{isoperimetric deficit}
of $\Omega$ as
\begin{equation}\label{deficit isoperimetric}
 \delta_\mu(\Omega)=P_\mu(\Omega)-J_\mu(\mu(\Omega)).
\end{equation} From the
inequalities \eqref{minorationOmegaD} and
\eqref{minorationOmegaC} of Theorem \ref{Main}, we can compute a lower bound on the
isoperimetric deficit as the
asymmetry
ranges from $0$ to its upper bound $\min(2\,m(\Omega),\,1-m(\Omega))$ (see
Lemma \ref{Domain}). Define the \textbf{isoperimetric deficit
function} $K_\mu$ as follows. 
\begin{enumerate}
  \item[$\bullet$] On $0<y\leq x\leq 1/2$, set $K_\mu(x,\,y)=J_\mu\left(x
-y/2\right)-J_\mu\left(x\right)+J_\mu\left(y/2\right)$.

  \item[$\bullet$] On $0<x\leq1/2$ and $x<y\leq\min(2x,\,1-x)$, set
\begin{equation*}\label{K mu 2}
 K_\mu(x,\,y)=J_\mu\left(x
+y/2\right)-J_\mu\left(x\right)+J_\mu\left( y/2\right).
\end{equation*}
\end{enumerate}
\noindent The isoperimetric deficit function $K_\mu(x,y)$ is defined on the domain of all
the possible values of $(m(\Omega),\lambda(\Omega))$. The isoperimetric
deficit function gives the lower bounds found in Theorem \ref{Main}. The next lemma focuses on the
variations of $K_\mu$.
\begin{lemma}\label{variation}
Let $0<x\leq1/2$. The function $y\mapsto K_\mu(x,\,y)$
is a non-decreasing lower semi-continuous function. Besides, it is concave on
$x<y\leq\min(2x,\,1-x)$.
\end{lemma}
\begin{proof}The proof is essentially based on the concavity of $J_\mu$.
\begin{description} 
  \item[\textbf{On $0<y\leq x$}] Let $\Psi(t)=1/2\left(J_\mu\left(x/2
-t\right)+J_\mu\left(x/2+t\right)\right)$. Then the point $(x/2,\,\Psi(t))$ is the middle of
the chord joining $(x/2-t,\,J_\mu(x/2-t))$ and $(x/2+t,\,J_\mu(x/2+t))$. Since $J_\mu$ is concave,
it is well known that $\Psi$ is a non-increasing function. Remark that
$K_\mu(x,y)=2\,\Psi(x/2-y/2)-J_\mu\left(x\right)$, thus $y\mapsto K_\mu(x,\,y)$ is
non-decreasing. Moreover the function is continuous as sum of continuous functions.
  \item[On $x<y\leq\min(2x,\,1-x)$] The function $y\mapsto K_\mu(x,\,y)$ is clearly concave
as sum of two concave functions (thus continuous). On this domain,
\[\left(y/2\right)+\left(x+y/2\right)=x+y\leq
x+\min(2x,\,1-x)\leq 1.\]
Hence the interval $\omega_y=(F^{-1}(y/2),F^{-1}(x+y/2))$ is on the left
of the origin. Remark that $K_\mu(x,y)=P_\mu(\omega_y)-J_\mu\left(x\right)$. The shifting
lemma (Lemma \ref{shiftinglemma}) applies here and shows that the
function $y\mapsto K_\mu(x,\,y)$ is non-decreasing (as $y$ increases, $\omega_y$ shifts to
the right).
\end{description}
The variation at $x$ is given by $K_\mu\left(x,x^+\right)-K_\mu(x,x)=J_\mu\left(3/2\,
x\right)-J_\mu\left(x/2\right)$, where $K_\mu\left(x,x^+\right)=\lim_{y\to
x^+}K_\mu\left(x,y\right)$. One can check that $\abs{1/2-x/2}\geq\abs{1/2-3x/2}$.
Using the
symmetry with respect to $1/2$ and the concavity of $J_\mu$, one can check that
$J_\mu\left(3/2\,x\right)\geq J_\mu\left(x/2\right)$. Hence
$K_\mu\left(x,x^+\right)\geq K_\mu(x,x)$.

\noindent This discussion shows that $y\mapsto K_\mu(x,\,y)$ is non-decreasing and lower
semi-continuous on the whole domain. This ends the proof.
\end{proof}
\noindent Defined the generalized inverse function of $y\mapsto K_\mu(x,\,y)$ as
\begin{equation*}
 K^{-1}_{\mu,\,x}(d)=\sup\left\{y\ |\ 0\leq
y\leq\min(2x,\,1-x)\ \mathrm{and}\ K_\mu(x,y)\leq d\right\}.
\end{equation*}
Lemma \ref{variation} shows that $y\mapsto K_\mu(x,\,y)$ is a non-decreasing lower
semi-continuous function. It is easy to check that $K^{-1}_{\mu,\,x}$ is non-decreasing.
Theorem \ref{Main} gives the next proposition.
\begin{proposition}\label{useless}
Let $\Omega$ be a measurable set and $\lambda(\Omega)$ be the asymmetry of $\Omega$.
Let $m(\Omega)=\min\left\{\mu(\Omega)\,,\ 1-\mu(\Omega)\right\}$ and
$\delta_\mu(\Omega)=P_\mu(\Omega)-J_\mu(\mu(\Omega))$. It holds,
\begin{equation}\label{lower bound deficit}
 \delta_\mu(\Omega)\geq
K_\mu(m(\Omega),\,\lambda(\Omega))\geq 0\,,
\end{equation}
\begin{equation}\label{upper bound lambda K mu}
 \mathrm{and}\quad \lambda(\Omega)\leq
K^{-1}_{\mu,\,m(\Omega)}(\delta(\Omega)).
\end{equation}
\end{proposition}
\begin{proof}
Since the asymmetry, the perimeter, the isoperimetric deficit, and $m(\Omega)$ are
complement invariant, suppose that $m(\Omega)=\mu(\Omega)\leq1/2$. Theorem \ref{Main} gives
\eqref{lower bound deficit}. Set $x=m(\Omega)$, the upper bound \eqref{upper bound lambda K
mu} is a consequence of the definition of $K^{-1}_{\mu,\,x}$ and \eqref{lower bound deficit}.
\end{proof}

\subsubsection{The Gaussian case}
We focus here on the Gaussian measure $\gamma$. Observe that
$J_\gamma(t)\mathop{\sim}\limits_{t\to0}
t\,\sqrt{2\log\left(1/t\right)}$, so that $K_\gamma(x,y)
\mathop{\sim}\limits_{y\to 0}
J_\gamma\left(\frac y2\right)\mathop{\sim}\limits_{y\to
0}\frac y2\,\sqrt{2\log\left(2/y \right)}$. In particular, there exists a constant
$C(x)$ that depends only on $x$ such that 
\[K_\gamma(x,y)\geq \frac y{C(x)}\,\sqrt{\log\left(1/y \right)},\quad \mathrm{with}\ 0\leq
y\leq\min(2x,\,1-x), \]
we recover \eqref{result CFMP} from Proposition \ref{useless}.

\subsubsection{Lower bound on the isoperimetric deficit}
In this section we give a convenient lower bound on the
isoperimetric deficit.  Define the function $L_\mu$ as follows. 
\begin{enumerate}
  \item[$\bullet$] On $0<y\leq x\leq 1/2$, set $L_\mu(x,\,y)=J_\mu\left( y/2\right)-
y/({2x})\,J_\mu\left(x\right)$.
  \item[$\bullet$] On $0<x\leq1/2$ and $x<y\leq\min(2x,\,1-x)$, set
\begin{equation*}\label{L mu 2}
 L_\mu(x,\,y)=J_\mu\left(y/2\right)-y/({2(1-x)})\,J_\mu\left(x\right).
\end{equation*}
\end{enumerate}
The next lemma shows that $0\leq L_\mu\leq K_\mu$.
\begin{lemma}\label{convenient}
Let $\Omega$ be a measurable set and $\lambda(\Omega)$ be the asymmetry of
$\Omega$. Let $m(\Omega)=\min\left\{\mu(\Omega)\,,\ 1-\mu(\Omega)\right\}$ and
$\delta_\mu(\Omega)=P_\mu(\Omega)-J_\mu(\mu(\Omega))$. 
It holds,
\begin{equation}\label{convenient lower bound deficit}
 \delta_\mu(\Omega)\geq
L_\mu(m(\Omega),\,\lambda(\Omega))\geq 0\,,
\end{equation}
\end{lemma}
\begin{proof}
Since the asymmetry, the perimeter, the isoperimetric deficit, and $m(\Omega)$ are
complement invariant, suppose that $m(\Omega)=\mu(\Omega)\leq1/2$. Set $x=m(\Omega)$ and
$y=\lambda(\Omega)$.
\begin{description} 
  \item[\textbf{On $0<y\leq x$}] Set $t=y/(2x-y)$ then $x-y/2=t\,y/2+(1-t)\,x$. Since
$J_\mu$ is concave, it holds
\begin{equation*}
 \begin{array}{rcl}
  K_\mu(x,y) & = & J_\mu\left(x
-\frac y2\right)-J_\mu\left(x\right)+J_\mu\left(\frac y 2\right)\geq (1+t)J_\mu\left(\frac
y 2\right) - tJ_\mu\left(x\right), \\ & = & \frac{1}{1-y/2x}\left(
J_\mu\left(\frac y 2\right) -
\frac y{2x}\,J_\mu\left(x\right)\right),\\
 & \geq &  J_\mu\left(\frac y 2\right) -
\frac y{2x}\,J_\mu\left(x\right).
 \end{array}
\end{equation*}
As $J_\mu$ is concave, the function $t\mapsto J_\mu(t)/t$ is non-increasing and thus
$(2/y)J_\mu(y/2) -
(1/x)J_\mu\left(x\right)\geq 0$.
  \item[On $x<y\leq\min(2x,\,1-x)$] Using symmetry with respect to $1/2$, remark that
\begin{equation*}
K_\mu(x,y) =  J_\mu\left(x
+\frac y2\right)-J_\mu\left(x\right)+J_\mu\left(\frac y2\right) =  J_\mu\left((1-x)
-\frac y2\right)-J_\mu\left(1-x\right)+J_\mu\left(\frac y2\right)
\end{equation*}
Substituting $x$ with $1-x$, the same calculus as above can be done.
\end{description} This ends the proof.
\end{proof}
\noindent  The lower bound given in Lemma \ref{convenient} is the key tool of the proof of the
continuity theorem.

\subsection{The continuity theorem}
In the following, we improve the lower bound \eqref{upper bound lambda K mu} on the asymmetry. We
begin with a
remark. Consider the \textit{exponential} case where
\[J_{\exp{}}(t)=t\,1\!\!1_{[0,1/2]}+(1-t)\,1\!\!1_{[1/2,1]}.\]
One gets $K_{\exp{}}=0$ on $0\leq y\leq x\leq 1/2$.
This means that there exists sets with a positive asymmetry and an isoperimetric
deficit null. In the case of \textit{the exponential-like distributions} (defined later
on), the intervals $(-\infty, F^{-1}(r))$ and $(F^{-1}(1-r),+\infty)$ are not the
only sets with minimal perimeter (up to a set of measure equals to $0$) given measure $r$. 

We specify this thought defining a natural hypothesis (\ref{lower
additivity}). Furthermore, we prove that the asymmetry goes to zero as the isoperimetric deficit
goes to
zero.
\subsubsection{The hypothesis \ref{lower additivity}}
We can get a better estimate on the asymmetry making another hypothesis. From now, suppose
that the measure $\mu$ is such that
\begin{equation}\label{lower additivity}\tag{$\mathcal H$}
\exists\, \varepsilon>0\quad \mathrm{s.t.}\quad t\mapsto {J_\mu(t)}/t\ \mathrm{is\ decreasing\
on}\ (0,\varepsilon).
\end{equation}
This hypothesis means that $J_\mu$ is non-linear in a neighborhood of the origin. We can be
more specific introducing the property:
\begin{equation}\label{exponential profile}\tag{$\overline{\mathcal H}$}
\exists\, \varepsilon>0\ \mathrm{and}\ c>0\quad \mathrm{s.t.}\quad J_\mu(t)= c\,t,\
\forall\,t\in[0,\varepsilon].
\end{equation}
Since $t\mapsto J_\mu(t)/t$ is non-increasing, it is not difficult to check that
\eqref{exponential profile} is the alternative hypothesis of \eqref{lower additivity}.

\subsubsection{The exponential-like case}
The exponential tail measures can be defined by the
following property:
\begin{equation}\label{exponential behavior}\tag{${\mathcal Exp}$}
\exists\, 
\tau>0\ \mathrm{and}\ c,c'>0\ \mathrm{s.t.}\ f(t)= c'\exp({ct}),\
\forall\,t\in(-\infty,\tau).
\end{equation}
\begin{proposition}
The property \eqref{exponential profile} is equivalent to the property \eqref{exponential
behavior}.
\end{proposition}
\begin{proof} 
The proof is essentially derived \cite{MR1327260} from the equality
$(F^{-1})'(t)={1}/{J_\mu(t)}$, for all $t\in(0,1)$. Suppose that the measure satisfies
\eqref{exponential profile}. Using the above equality
for sufficiently small values of $r$, one can check that $F^{-1}(r)= \frac1c\log(r)+c''$, where
$c''$ is a constant. Hence $F(x)=\exp(c(x-c''))=\frac{c'}c\exp(cx)$, which gives the property
\eqref{exponential behavior}. Conversely, suppose that the measure satisfies
\eqref{exponential behavior}. A simple
computation gives the property \eqref{exponential profile}.
\end{proof}
\noindent Suppose that $\mu$ satisfies \eqref{exponential profile}. It is not difficult to check
that the sets (and their symmetric)
$(-\infty,F^{-1}(r-s))\cup(F^{-1}(s),+\infty)$, for all $s\in(0,\min(\epsilon,r))$,
have minimal perimeter given measure $r$. It would
be natural to define the asymmetry with these sets (and not only $(-\infty,-\sigma)$ and
$(\sigma,+\infty)$). In this case, our asymmetry (defined by \eqref{asymmetry}) is not
relevant in terms of continuity.

\subsubsection{Continuity of the asymmetry for non-exponential distributions}
The hypothesis \eqref{lower additivity} ensures that the distribution is
non-exponential. It is the right framework dealing with continuity as shown in the next
theorem.
\begin{theorem}[Continuity]\label{continuity}
Assume that the measure $\mu$ satisfies the assumption \ref{lower additivity}, then the
asymmetry
goes to zero as the isoperimetric deficit goes to zero.
\end{theorem}
\begin{proof}
The proof is based on Lemma \ref{convenient} and Proposition \ref{useless}. Let
$u,v\in(0,1)$, define $\rho(u,v)={J_\mu(u)}/{u}-{J_\mu(v)}/{v}$. Suppose $u<v$. Since $J_\mu$ is
concave, it is easy to check that if $\rho(u,v)=0$, then
$\forall\,u'\leq u$, $\rho(u',v)=0$. In particular \ref{lower
additivity} implies that $\forall\,u<v\,,\ \rho(u,v)> 0$, for sufficiently small values of
$v$. Remark that
$L_\mu(x,y)=(y/2)\rho(y/2,x)$ if $0<y\leq x$, and $L_\mu(x,y)=(y/2)\rho(y/2,1-x)$ if
$x<y\leq\min(2x,\,1-x)$. Hence \ref{lower additivity} implies that $L_\mu>0$. Using Lemma
\ref{convenient}, it yields that $K_\mu>0$.

Finally, it is easy to check that if $K_\mu>0$ then there exists a neighborhood of $0$ such
that $K_\mu$ is increasing. Taking a sufficiently small neighborhood if necessary, one can
suppose that $K_\mu$ is continuous (the only point of discontinuity of $K_\mu$ is $y=x$). On
this neighborhood, $K_{\mu,\,x}^{-1}$ is a continuous increasing function. Using \eqref{upper
bound lambda K mu}, this gives the expected result.
\end{proof}
\noindent Roughly, it uncovers that a set of given measure and almost minimal boundary measure
is necessarily close to be a half-line. Moreover we recover the following well-known result.
\begin{cor}
Assume that the measure $\mu$ satisfies the assumption \ref{lower additivity}, then the half-lines
are the only sets of given measure and minimal perimeter $($up to a set of $\mu$-measure null$)$.
\end{cor}
\noindent This last results ensure that the asymmetry is the relevant notion speaking of
the isoperimetric deficit.
\bigskip

\noindent \textbf{Acknowledgments:} The author thanks F. Barthe for
invaluable discussions and unfailing support. We also thank F. Maggi for communicating useful
references to us.

\bibliographystyle{alpha}
\nocite{*}
\bibliography{QIIRL}

\end{document}